\documentclass[12pt, a4paper]{amsart}

\usepackage[hmargin=30mm, vmargin=25mm, includefoot, twoside]{geometry}
\usepackage[bookmarksopen=true]{hyperref}

\usepackage{amsfonts,amssymb,verbatim}
\usepackage{latexsym}
\usepackage{mathrsfs}
\usepackage{stmaryrd}
\usepackage{xspace}
\usepackage{enumerate, paralist}
\usepackage{graphicx}
\usepackage[all]{xy}
\usepackage{extarrows}

\usepackage[usenames,dvipsnames]{color}

\usepackage{amsthm}
\usepackage{amsmath}

\newtheorem{thm}{Theorem}[section]
 \newtheorem{cor}[thm]{Corollary}
 \newtheorem{lem}[thm]{Lemma}
 \newtheorem{prop}[thm]{Proposition}

\newtheorem{introthm}{Theorem}

 \newtheorem{introdefn}[introthm]{Definition}

\numberwithin{equation}{section}

 \theoremstyle{definition}
  \newtheorem{defn}[thm]{Definition}
  \newtheorem{question}[thm]{Question}
 \theoremstyle{remark}
 \newtheorem{rem}[thm]{Remark}
  \newtheorem{ex}[thm]{Example}

\def\B{\mathfrak B}

\def\A{\mathcal A}

\def\supp{\mathrm{supp}}
\def\diam{\mathrm{diam}}
\def\Id{\mathrm{Id}}

\def\N{\mathbb N}
\def\C{\mathbb C}
\def\Z{\mathbb{Z}}

\def\Nd{\mathcal N}

\def\Cq{C^*_{uq}}

\begin{document}

\title{quasi-local algebras and asymptotic expanders}

\author{Kang Li, Piotr Nowak, J\'{a}n \v{S}pakula and Jiawen Zhang}

\address[K. Li]{Institute of Mathematics of the Polish Academy of Sciences, \'{S}niadeckich 8, 00-656 Warsaw, Poland.}
\email{kli@impan.pl}
\address[P. Nowak]{Institute of Mathematics of the Polish Academy of Sciences, \'{S}niadeckich 8, 00-656 Warsaw, Poland.}
\email{pnowak@impan.pl}

\address[J. \v{S}pakula]{School of Mathematics, University of Southampton, Highfield, SO17 1BJ, United Kingdom.}
\email{jan.spakula@soton.ac.uk}

\address[J. Zhang]{School of Mathematics, University of Southampton, Highfield, SO17 1BJ, United Kingdom.}
\email{jiawen.zhang@soton.ac.uk}

\date{}
%\subjclass[2010]{20F65, 20F67, 20F69}
\keywords{Expanders, Nuclearity, Property A, Quasi-local algebras.}

\thanks{KL and PN were supported by the European Research Council (ERC) under the European Union's Horizon 2020 research and innovation programme (grant agreement no. 677120-INDEX). J\v{S} was partially supported by Marie Curie FP7-PEOPLE-2013-CIG Coarse Analysis (631945). JZ was supported by the Sino-British Trust Fellowship by Royal Society, International Exchanges 2017 Cost Share (China) grant EC$\backslash$NSFC$\backslash$170341, and NSFC11871342.}

% orig from Jiawen
%\baselineskip=16pt

\begin{abstract}
In this paper, we study the relation between the uniform Roe algebra and the
uniform quasi-local algebra associated to a metric space of bounded geometry. In
the process, we introduce a weakening of the notion of expanders, called
asymptotic expanders. We show that being a sequence of asymptotic expanders is a
coarse property under certain connectedness condition, and it implies non-uniformly local amenability. Moreover, we
also analyse some $C^*$-algebraic properties of uniform quasi-local algebras. In
particular, we show that a uniform quasi-local algebra is nuclear if and only if
the underlying metric space has Property A.
\end{abstract}

\date{\today}
\maketitle

\parskip 4pt

\noindent\textit{Mathematics Subject Classification} (2010): 46H35, 46L05, 20F65, 05C99.\\

\section{Introduction}
(Uniform) Roe algebras are $C^*$-algebras associated to discrete metric spaces, which reflect and encode the coarse (or large-scale) geometry of the underlying metric spaces. They have been well-studied and have fruitful applications, among which the most important ones would be the (uniform) coarse Baum-Connes conjecture, the Novikov conjecture, the zero-in-the-spectrum conjecture and the conjecture of positive scalar curvature on manifolds (e.g. \cite{STY02, MR2523336, Yu95, MR1451759, Yu97b, MR1626745, Yu00}).

Recently, there has been substantial research about the interplay between coarse-geometric properties of a metric space $X$ with bounded geometry and analytic properties of its uniform Roe algebra $C^*_u(X)$ (e.g. \cite{ALLW17, BL18, CL19, LL, LW18, STY02, MR2800923, WZ10}). A prototypical result in this direction comes from \cite{GK02,Oz00,STY02}: a metric space $X$ has Property A if and only if $C^*_u(X)$ is a nuclear $C^*$-algebra.

A fundamental question is to determine whether a given operator belongs to the uniform Roe algebra. To overcome this issue, Roe suggested the notion of quasi-locality in \cite{Roe88, Roe96} and observed that operators in uniform Roe algebras are always quasi-local. The converse is open, although it has been proven under additional assumptions on the underlying spaces \cite{Eng19,ST19,SZ18}. This piece revolves around comparing the uniform Roe algebra $C^*_u(X)$ of a bounded geometry metric space $X$ with the $C^*$-algebra $\Cq(X)$ of all quasi-local operators in $\B(\ell^2(X))$ (see Definition~\ref{def: quasi-locality} and Definition~\ref{def:quasi-loc algebra} for the definition of the \emph{uniform quasi-local algebra} $\Cq(X)$). We always have $C^*_u(X)\subseteq \Cq(X)$, and if the space $X$ has Property A, then we have the equality $C^*_u(X)=\Cq(X)$ \cite[Theorem 3.3(2)]{SZ18}.

The motivation for this paper is to look for \emph{obstructions} to this
equality. More precisely, we attempt to tell the difference between $C^*_u(X)$
and $\Cq(X)$ via the averaging projection $P_X\in\B(\ell^2(X))$ over the coarse
disjoint union $X=\bigsqcup_{n\in \N} X_n$ of a sequence of finite metric spaces
$\{X_n\}_{n \in \N}$ (see Definition~\ref{defn: averaging projection}). It is
well known that if $X$ is an expander, then $P_X\in C^*_u(X)$ (see also the
discussion before Definition \ref{defn: averaging projection}). On the other
hand, $P_X \notin C^*_u(X)$ if $X$ can be coarsely embedded into some Hilbert
space according to Finn-Sell's work \cite[Proposition~35]{Fin14}. Hence, it is
crucial to know when the averaging projection $P_X$ belongs to $\Cq(X)$. It
turns out that the quasi-locality of $P_X$ is equivalent to $\{X_{n}\}_{n\in\N}$
being a sequence of \emph{asymptotic expanders}, which is a slight weakening of
being an expander sequence.

\begin{introdefn}[Definition \ref{def: expanderish condition}]\label{introdefn: expanderish}
A sequence of finite metric spaces $\{X_n\}_{n \in \N}$ with $|X_n| \to \infty$ (as $n\to\infty$) is said to be a sequence of \emph{asymptotic expanders} if for any $\alpha>0$, there exist $c\in (0,1)$ and $R>0$ such that for any $n \in \N$ and $A \subseteq X_n$ with $\alpha|X_n| \leq |A| \leq |X_n|/2$, we have $|\partial_R A| > c|A|$, where $\partial_R A:=\{x\in X_n\backslash A: d(x,A)\leq R\}$.
\end{introdefn}
We prove the following statement:

\begin{introthm}[Theorem~\ref{prop: expanderish condition}]\label{introthm: expanderish iff ql}
Let $\{X_n\}_{n \in \N}$ be a sequence of finite metric spaces with $|X_n| \to \infty$ as $n\to\infty$. Let $X=\bigsqcup_{n\in\N}X_{n}$ be a coarse disjoint union, and let $P_X$ the averaging projection of the sequence $\{X_n\}_{n\in \N}$. Then
$P_X$ is quasi-local \emph{if and only if}
$\{X_n\}_{n\in \N}$ is a sequence of asymptotic expanders.
\end{introthm}

\begin{rem}\label{rem:work-for-unbounded-geometry}
  While we assume bounded geometry throughout the paper, the notions of quasi-locality and asymptotic expanders themselves are meaningful also when $X$ does not have bounded geometry. Our proof of the above theorem does not require bounded geometry, and thus in Subsection \ref{subsec:asymptotic-expanders} we only assume that the metric spaces are discrete.
\end{rem}

Therefore, the existence of a sequence of asymptotic expanders whose coarse disjoint union $X$ can be coarsely embedded into some Hilbert space would imply that the associated averaging projection $P_X$ is quasi-local but does not belong to the uniform Roe algebra of $X$. In other words, $C^*_u(X)\subsetneqq \Cq(X)$ (see Proposition~\ref{cor: non CE}). However, we did not yet succeed in finding such an example of $X$ (see Question~\ref{ques}).

Asymptotic expanders themselves might be of independent interest to experts in
graph theory (see Theorem~\ref{prop: expanderish condition} for different
formulations similar to the Cheeger constant of expanders). We show that
asymptotic expanders are strictly more general than expanders (see
Corollary~\ref{cor: expanderish but non-expander}). Moreover, we study coarse
properties of asymptotic expanders, showing that being a sequence of asymptotic
expander graphs is invariant under coarse equivalences, and is incompatible with
uniformly local amenability:

\begin{introthm}[Corollary \ref{cor:coarse-invariance-ish-when-connected}]\label{introthm: coarse equivalence}
Let $\{X_n\}_{n \in \N}$ and $\{Y_n\}_{n \in \N}$ be sequences of finite connected
graphs with bounded valency, such that $|X_n|, |Y_n| \to \infty$ as $n\to
\infty$.
If coarse disjoint unions $X=\bigsqcup_{n\in\N}X_{n}$ and $Y=\bigsqcup_{n\in\N}Y_{n}$ are coarsely equivalent, and $\{X_{n}\}_{n\in\N}$ is a sequence of asymptotic expanders, then so is $\{Y_{n}\}_{n\in\N}$.
\end{introthm}

\begin{introthm}[Theorem \ref{thm: expanderish implies non ULA}]\label{introthm: non ula}
Let $X$ be a metric space with bounded geometry, which is a coarse disjoint union of a sequence of asymptotic expanders. Then $X$ is not uniformly locally amenable. In particular, $X$ does not have Property A.
\end{introthm}

Finally, we study $C^*$-algebraic properties of the uniform quasi-local algebra $\Cq(X)$ of a metric space with bounded geometry. As alluded to above, we already know that $X$ having Property A is equivalent to the nuclearity of its uniform Roe algebra $C^*_u(X)$, and when $X$ has Property A we have $C^*_u(X)=\Cq(X)$. Concerning uniform quasi-local algebras, we are able to prove the following theorem:

\begin{introthm}[Theorem~\ref{thm: nuclearity of quasi-local algebra} and Proposition~\ref{prop: two algebras are euqal}]\label{introthm: nuclearity}
Let $X$ be a metric space of bounded geometry. Then
\begin{itemize}
\item $X$ has Property A if and only if the uniform quasi-local algebra $\Cq(X)$ is nuclear;
\item $C^*_u(X)=\Cq(X)$ if and only if $\ell^\infty(X)\subseteq \Cq(X)$ is a Cartan subalgebra.
\end{itemize}

\end{introthm}
In particular, Theorem~\ref{introthm: nuclearity} shows that we can \emph{not} use nuclearity to distinguish $\Cq(X)$ from $C^*_u(X)$. Moreover, we know that $\ell^\infty(X)\subseteq C_u^*(X)$ is always a Cartan subalgebra, and structural and uniqueness questions for Cartan subalgebras in uniform Roe algebras were intensively studied in \cite{WW18}.

The paper is organised as follows. In Section \ref{sec:preliminaries}, we recall some basic notions in coarse geometry which are used throughout the paper. In Section \ref{sec:expanderish}, we introduce the notion of asymptotic expanders, and prove Theorem \ref{introthm: expanderish iff ql} and Theorem \ref{introthm: coarse equivalence}. We provide a proof of Theorem~\ref{introthm: non ula} in Section~\ref{sec:ula}. Moreover, Section~\ref{sec:nuclearity} and Section~\ref{sec:cartan} are devoted to Theorem~\ref{introthm: nuclearity}. Finally, we raise several open questions in Section \ref{sec:question}.

\emph{Acknowledgments.} We would like to thank Rufus Willett for sharing a draft on the quasi-locality of averaging projections. We would also like to thank Hiroki Sako for bringing Example \ref{Ex: non-expander but expanderish} to our attention, and Baojie Jiang for several illuminating discussions. Finally, we would like to thank the anonymous referee for pointing a mistake in the statement of Theorem \ref{thm:coarse-invariance-of-ish} of an early version, and many useful suggestions to make the paper more readable.

\section{Preliminaries}\label{sec:preliminaries}

Let $(X,d)$ be a metric space, $x\in X$ and $R>0$. Denote $B(x,R)$ the
closed ball in $X$ with centre $x$ and radius $R$. For any $A \subseteq X$,
denote by $|A|$ the cardinality of $A$, $\Nd_R(A)=\{x\in X : d(x,A)\leq R\}$ the \emph{$R$-neighbourhood of $A$}, and $\partial_R A=\{x\in X\backslash A: d(x,A)\leq R\}$ the \emph{(outer) $R$-boundary of $A$}. Recall that a metric space $(X,d)$ has \emph{bounded geometry} if $\sup_{x\in X}|B(x,R)|$ is finite for each $R>0$; \footnote{It is worth noticing that a metric space with bounded geometry must be discrete.} we shall occasionally use the notation $N_X(R)=\sup_{x\in X}|B(x,R)|$. We say that $X$ is \emph{$D$-connected} (for some $D\geq0$), if for all $x,y\in X$ there exists a sequence $x=x_{0}, x_{1}, \dots, x_{n}=y$ of points in $X$, such that $d(x_{i-1},x_{i})\leq D$ for all $i\in\{1,\dots,n\}$.

For $A \subseteq X$, we use $\chi_A$ for the \emph{characteristic function} of $A$ and the symbol $\delta_x$ for $\chi_{\{x\}}$ for $x\in X$. By a slight abuse of notation, we shall use the functions $\chi_A$ both as vectors in $\ell^2(X)$ (when $A$ is finite), and as multiplication operators on $\ell^2(X)$ in the sense below defined.

\emph{Throughout this paper, unless stated otherwise, $(X, d)$ denotes a metric
space with bounded geometry. The notable exception is Subsection \ref{subsec:asymptotic-expanders}, where we only require $X$ to be discrete.}

\subsection{Uniform Roe algebras} \label{subsect:unif-Roe-algs}

An operator $T\in \B(\ell^2 (X))$ can be viewed as an $X$-by-$X$ matrix $[T_{x,y}]_{x,y \in X}$ with $T_{x,y}=\langle T \delta_y, \delta_x \rangle\in \C$. We say that $T\in \B(\ell^2 (X))$ has \emph{finite propagation} if there exists some constant $R >0$ such that $T_{x,y}=0$ if $d(x,y)>R$. The smallest number $R$ satisfying this condition is called the \emph{propagation of $T$}.

There are two elementary classes of operators with finite propagation: multiplication operators and partial translations: let $f \in \ell^\infty(X)$, the pointwise multiplication provides an operator in $\B(\ell^2(X))$ with zero propagation, called the \emph{multiplication operator} of $f$ and still denoted by $f$ for simplicity. For the latter, let $D,R\subseteq X$ and $\theta: D \to R$ be a bijection. Define a matrix $V^\theta$ by
\begin{equation*}
V^{\theta}_{x,y}=
    \begin{cases}
           ~1, & x\in D \mbox{~and~} y=\theta(x), \\
           ~0, & \mbox{otherwise}.
     \end{cases}
\end{equation*}
If $\sup_{x\in D}d(x,\theta(x))$ is finite, then $V^\theta$ is a partial isometry in $\B(\ell^2(X))$ with finite propagation and $V^\theta$ is called a \emph{partial translation (operator)}.
It is direct to check that the set of all finite propagation operators in $\B(\ell^2 (X))$ forms a $\ast$-algebra, called the \emph{algebraic uniform Roe algebra} $\C_u [X]$ of $X$. It is well known that $\C_u [X]$ is generated by multiplication operators and partial translations as a $*$-algebra for every metric space $X$ with bounded geometry (see \cite[Lemma~4.27]{Roe03}). The \emph{uniform Roe algebra} $C^*_u (X)$ of $X$ is the operator-norm closure of $\mathbb{C}_u [X]$ inside $\B(\ell^2(X))$.

\subsection{Quasi-locality}

By definition, an operator in $\B(\ell^2(X))$ belongs to the uniform Roe algebra $C^*_u(X)$ if and only if it can be approximated by finite propagation operators in norm, which is usually not easy to check in practice. In order to find a more intrinsic and practical approach to characterise elements in $C^*_u(X)$, Roe introduced the following notion of quasi-locality.

\begin{defn}[\cite{Roe88, Roe96}]\label{def: quasi-locality}
Let $R,\varepsilon > 0$. An operator $T \in \B(\ell^2(X))$ is said to have $(R,\varepsilon)$-\emph{propagation} if for any $A, B\subseteq X$ such that $d(A,B) \geq R$, we have	
\[
\| \chi_A T\chi_B\|\leq \varepsilon.
\]
We say that $T$ is \emph{quasi-local}, if for all $\varepsilon > 0$, there exists $R > 0$ such that $T$ has $(R,\varepsilon)$-\emph{propagation}.
\end{defn}

It is routine to check that the set of all quasi-local operators in $\B(\ell^2(X))$ forms a $C^*$-subalgebra of $\B(\ell^2(X))$. Hence, we make the following definition:
\begin{defn}\label{def:quasi-loc algebra}
Let $(X,d)$ be a discrete metric space. The \emph{uniform quasi-local algebra} of $X$, denoted by $\Cq(X)$, is defined to be the $C^*$-algebra of all quasi-local operators in $\B(\ell^2(X))$.
\end{defn}

It is clear that operators with finite propagation are quasi-local. Hence, $C^*_u(X) \subseteq \Cq(X)$ after taking the closure.

\subsection{Comparing $C^*_u(X)$ with $\Cq(X)$}

We already noticed that $C^*_u(X) \subseteq \Cq(X)$ holds generally. For the opposite inclusion, the best existing result in this direction provides only a sufficient condition, which is Property A. Property A was introduced by Yu in \cite{Yu00} in his study of the coarse Baum-Connes conjecture and the Novikov conjecture. Here we recall some of equivalent characterisations of Property A and one of them is in terms of ghost operators: an operator $T \in \B(\ell^2(X))$ is called a \emph{ghost operator} if for any $\varepsilon>0$, there exists a bounded subset $B \subseteq X$ such that for any $x,y \in X \setminus B$, we have $|T_{x,y}|< \varepsilon$.

\begin{prop}[{\cite[Theorem 1.2.4]{Wil09b}}, {\cite[Theorem 1.3]{RW14}}]\label{prop: characterisation of property A}
Let $(X, d)$ be a metric space with bounded geometry. Then the following are equivalent:
\begin{enumerate}
  \item $(X, d)$ has Property A.
  \item For any $R,\varepsilon > 0$ there exist a map $\xi : X \to \ell^2(X)$, and a number $S>0$ such that:
     \begin{enumerate}
     \item $\|\xi_x\|_2=1$ for every $x \in X$;
     \item if $d(x, y) < R$, then $\|\xi_x-\xi_y\|_2 < \varepsilon$;
     \item $\supp(\xi_x)\subseteq B(x,S)$ for every $x \in X$.
     \end{enumerate}	
  \item All ghost operators in $C^*_u(X)$ are compact.
\end{enumerate}	
\end{prop}

Finally, we recall the main result from \cite{SZ18} that $C^*_u(X)=\Cq(X)$ provided the space $X$ has Property A:
\begin{prop}\cite[Theorem~3.3]{SZ18}\label{prop: SZ result}
Let $X$ be a metric space with bounded geometry. If $X$ has Property A, then $C^*_u(X)=\Cq(X)$.
\end{prop}

\section{asymptotic expanders} \label{sec:expanderish}
In this section, we begin by recalling the notion of expander graphs and the
averaging projection over a sequence of finite metric spaces. Then we introduce
the notion of asymptotic expanders (see Definition \ref{introdefn:
  expanderish}), which has close relation with the associated averaging
projection and the uniform quasi-local algebra. Moreover, we show that being a
sequence of asymptotic expanders is a coarse property under certain mild conditions (see Theorem~\ref{thm:coarse-invariance-of-ish} and Corollary~\ref{cor:coarse-invariance-ish-when-connected}).

By a \emph{graph} in this paper we always mean an undirected graph with no loops
and multiple edges, in the sense of graph theory. Let us fix further terminology:
Let $X=(V,E)$ be a connected graph. The vertex set $V$ is regarded as a metric space equipped with the edge-path metric $d$. By a slight abuse of notation, this is the metric space we refer to when we regard $X$ as a metric space. We say that a graph $X$ has \emph{bounded valency} if there exists some $k \in \N$ such that for any vertex $x\in V$, there are at most $k$ vertices connecting to $x$. It is clear that $X$ has bounded valency \emph{if and only if} $(V,d)$ has bounded geometry. We set $\partial A:=\partial_1 A$ to denote the 1-boundary of $A$.

\subsection{Expander graphs}
Recall that expander graphs are sequences of finite graphs which are highly connected but sparse at the same time. The first explicit construction was due to Margulis \cite{Mar73} using Kazhdan's property (T).

\begin{defn}\cite[Definition~5.6.2]{NY12}\label{defn: expander graphs}
Let $\{X_n=(V_n,E_n)\}_{n \in \N}$ be a sequence of finite graphs with bounded valency and $|V_n| \to \infty$ as $n\to\infty$. $\{X_n\}_{n\in \N}$ is said to be \emph{a sequence of expander graphs} if there exists some $c>0$ such that for any $n \in \N$ and $A \subseteq V_n$ with $1\leq |A| \leq |V_n|/2$, then $|\partial A| > c|A|$.
\end{defn}

\begin{rem}
The expanding condition implies that all the graphs $X_n$ in an expander sequence are connected (equivalently, $1$-connected as metric spaces).
\end{rem}

Alternatively, we have the following analytic characterisation of expander graphs (for a proof, see \cite[{Theorem~1.2.3}]{MR1989434}, or additionally Section 5.6 of \cite{NY12}):
\begin{prop}\label{prop: Poincare expander graphs}
Let $X=\{X_n=(V_n,E_n)\}_{n \in \N}$ be a sequence of finite graphs with bounded valency and $|V_n| \to \infty$ as $n\to\infty$. Then $X$ is a sequence of expander graphs \emph{if and only if} there exists some $c>0$ such that for any $n \in \N$ and any $f: V_n \to \C$ such that $\sum_{x\in V_n}f(x)=0$, the following Poincar\'{e} Inequality holds:
\begin{equation}\label{Eq: Poincare Inquality}
\sum_{x,y \in V_n;~d(x,y)=1} |f(x)-f(y)|^2 \geq c\sum_{x\in V_n}|f(x)|^2.
\end{equation}
\end{prop}

Recall that the \emph{discrete Laplacian} $\Delta_Y$ of a graph $Y=(V,E)$ is the $V$-by-$V$ matrix, with valencies of vertices on the diagonal; $-1$ at $(x,y)$-entry whenever there is an edge connecting $x$ and $y$; and 0 otherwise. For a sequence of graphs as in the proposition above, we denote by $\Delta$ the $X$-by-$X$ block-diagonal matrix with blocks being $\Delta_{X_n}$. This defines a bounded operator on $\ell^2(X)$ (because of the bounded valency) of propagation $1$.

A standard computation shows that the condition in Proposition \ref{prop: Poincare expander graphs} says that the discrete Laplacian $\Delta$ has a spectral gap, i.e., there exists $c>0$ such that $\sigma(\Delta) \subseteq \{0\} \cup [c, \infty)$. Hence the characteristic function $\chi_{\{0\}}$ of the set $\{0\}\subset \mathbb{R}$ is continuous on the spectrum $\sigma(\Delta)$ of the operator $\Delta$. Thus we can apply continuous functional calculus, and obtain that $\chi_{\{0\}}(\Delta)$ is in the $C^*$-algebra generated by $\Delta$, which is contained in the uniform Roe algebra $C^*_u(X)$ (cf.~discussion after~\cite[Proposition 3.1]{RW14}). The operator $\chi_{\{0\}}(\Delta)$ is the projection onto the kernel of $\Delta$, which consists of functions constant on each $X_n$. This projection admits another description as the averaging projection $P_X$ on $X$, which is defined as follows:

\begin{defn}\label{defn: averaging projection}
Let $Y$ be a set and $F$ be a finite subset of $Y$. The \emph{averaging projection of $F$}, denoted by $P_F$, is the orthogonal projection onto the span of $\chi_F \in \ell^2(Y)$. In the matrix form, it can be represented by:
\begin{equation*}
(P_F)_{x,y}=
\begin{cases}
  ~1/|F|, & x,y \in F, \\
  ~0, & \mbox{otherwise}.
\end{cases}
\end{equation*}
Given a sequence of finite sets, let $X=\bigsqcup_{n\in\N}X_{n}$ as a set. Define the \emph{averaging projection of the sequence $\{X_{n}\}_{n\in\N}$} to be
$$P_X:=\sum_{n \in \N} P_{X_n} \in\B(\ell^{2}(X)),$$
which converges in the strong operator topology on $\B(\ell^2(X))$. (We are indulging in a slight abuse of notation here, cf.~Remark \ref{rem:avg-projn-depends-on-decomposition} below.)

Given a sequence of finite metric spaces $\{(X_n,d_n)\}_{n \in \N}$,
we endow $X=\bigsqcup_{n\in\N}X_{n}$ with a metric: we say that $(X,d)$ is a \emph{coarse disjoint union} of $\{X_n\}_{n\in\mathbb{N}}$, if the metric $d$ on each $X_n$ agrees with $d_n$ and satisfies:
$$d(X_n, X_m) \to\infty\text{ as }n+m\to\infty\text{ and }m\not=n.$$
Note that such a metric $d$ is well defined up to coarse equivalence. Whenever we need a concrete choice, we shall take $d$ such that for all $m,n\in\mathbb{N}$, $m\not=n$ and $x\in X_n$, $y\in X_m$ we have $d(x,y)=n+m+\diam(X_n)+\diam(X_m)$.
For any such metric, $P_{X}\in\B(\ell^{2}(X))$ is a non-compact ghost projection.
\end{defn}

\begin{rem}\label{rem:avg-projn-depends-on-decomposition}
The projection $P_{X}$ depends on the choice of the sequence $\{X_{n}\}_{n\in\N}$, not only on the (bijective) coarse equivalence type of a coarse disjoint union $X$ itself. For a simple example, consider the sequence with each $X_{n}=\{x_{n}\}$ consisting of a single point. First, we endow $X=\bigsqcup_{n\in\N}X_{n}$ with the metric as above; then $P_{X}=\Id\in\B(\ell^{2}(X))$. However, if we denote $Y_{k}=\{x_{2k-1},x_{2k}\}$ and declare $d(x_{2k-1},x_{2k})=k$ for $k\in\N$, then $Y=\bigsqcup_{k\in\N}Y_{k}$ is bijectively coarsely equivalent to $X$ (as sets, $X=Y$), but $P_{Y}\not=\Id\in \B(\ell^{2}(X))$.
\end{rem}

From the discussion before Definition \ref{defn: averaging projection}, we obtain the following:
\begin{cor}\label{cor: expander averaging proj in uniform Roe algebra}
Let $\{X_n\}_{n \in \N}$ be a sequence of expander graphs, and let $X=\bigsqcup_{n\in\N}X_{n}$ be a coarse disjoint union. Then the averaging projection $P_X$ belongs to the uniform Roe algebra $C^*_u(X)$.
\end{cor}

The following example is implicitly suggested in \cite[Proposition 2.4]{Wan07} and it was brought to our attention by Sako. It shows that the converse does not hold in general.

\begin{ex}\label{Ex: non-expander but expanderish}
Let $X$ be the coarse disjoint union of a sequence of expander graphs $\{X_n\}_{n \in \N}$ with bounded valency at most $k$. For any $n \in \N$, choose an arbitrary finite graph $F_n$ of degree at most $k$, satisfying $|F_n| \to \infty$ and $|F_n|/|X_n|\to 0$ as $n\to\infty$. Here we regard $X_n$ and $F_n$ as metric spaces with the edge-path metrics.

For each $n \in \N$, we construct a new graph $Y_n$ which is the disjoint union of $X_n$ and $F_n$ except that one additional edge is attached between two chosen vertices $x_n$ in $X_n$ and $y_n$ in $F_n$. Clearly, $Y_n$ is a finite graph of valency at most $k+1$. We claim that $\{Y_n\}_{n \in \N}$ is \emph{not} a sequence of expander graphs, but the averaging projection $P_Y$ belongs to the uniform Roe algebra $C^*_u(Y)$, where $Y$ is the coarse disjoint union of $\{Y_n\}_{n \in \N}$.

Indeed, since $|F_n|/|X_n| \to 0$, we can take a sufficiently large $n$ such that $|F_n| \leq |Y_n|/2$. By construction, $\partial F_n=\{x\in X_n: d(x,F_n)=1\} = \{x_n\}$, which implies that $|\partial F_n|/|F_n| \to 0$. Hence, $Y=\{Y_n\}_{n \in \N}$ is not a sequence of expander graphs.

Now we show that the averaging projection $P_Y$ belongs to the uniform Roe algebra $C^*_u(Y)$. In fact, this follows directly from \cite[Proposition 2.4]{Wan07}. For convenience of the readers, we provide a proof here. Since $X$ is a subspace of $Y$, we have $P_X \in \B(\ell^2(X)) \subseteq \B(\ell^2(Y))$. We claim that the difference $P_Y-P_X$ is a compact operator in $\B(\ell^2(Y))$. In fact, a direct calculation shows that for each $n$ and $x,y \in Y_n$:
\begin{equation*}
(P_{Y_n}-P_{X_n})_{x,y}=
\begin{cases}
  ~-\frac{|F_n|}{|X_n|(|X_n|+|F_n|)}, & x,y \in X_n; \\
  ~\frac{1}{|X_n|+|F_n|}, & \mbox{otherwise}.
\end{cases}
\end{equation*}
Since each operator $P_{Y_n}-P_{X_n}$ is represented by a finite matrix, its operator norm does not exceed its Frobenius norm:
\begin{eqnarray*}
\|P_{Y_n}-P_{X_n}\|_F^{2} &=& \sum_{x,y \in Y_n}\left|(P_{Y_n}-P_{X_n})_{x,y}\right|^2\\
 &=&  \sum_{(x,y) \in X_n^2}\frac{|F_n|^2}{|X_n|^2(|X_n|+|F_n|)^2} + \sum_{(x,y) \in Y_n^2\setminus X_n^2}\frac{1}{(|X_n|+|F_n|)^2} \\
 &=& \frac{|X_n|^2\cdot |F_n|^2}{|X_n|^2(|X_n|+|F_n|)^2} + \frac{(|X_n|+|F_n|)^2 - |X_n|^2}{(|X_n|+|F_n|)^2} \\
 &=& \frac{2|X_n|\cdot|F_n| + 2|F_n|^2}{|X_n|^2 + 2|X_n|\cdot|F_n| + |F_n|^2}.
\end{eqnarray*}
By the assumption that $|F_n|/|X_n| \to 0$, we have $\|P_{Y_n}-P_{X_n}\| \to 0$ as $n \to \infty$. Hence $P_Y-P_X=\sum_{n \in \N} (P_{Y_n}-P_{X_n})$ converges in the operator norm. Since each block $P_{Y_n}-P_{X_n}$ has finite rank, it is clear that $P_Y-P_X$ is a compact operator. From Corollary \ref{cor: expander averaging proj in uniform Roe algebra}, $P_X \in C^*_u(X) \subseteq C^*_u(Y)$, which implies $P_Y \in C^*_u(Y)$ as required.
\end{ex}

\subsection{Asymptotic expanders}\label{subsec:asymptotic-expanders}
Example \ref{Ex: non-expander but expanderish} shows that the property of being a sequence of expander graphs \emph{cannot} be characterised by the condition that the averaging projection belongs to the uniform Roe algebra. However, the counterexample is just a slight deformation of expanders.

In this section, we explore when the averaging projection $P_X$ is quasi-local, and introduce the notion of asymptotic expanders. We start with some elementary calculations.

\begin{lem}\label{lem: averaging proj calculation}
Let $X$ be a discrete metric space, $F$ be a finite subset of $X$ and $A,B \subseteq F$. Then
$$\|\chi_A P_F \chi_B\|=\frac{\sqrt{|A||B|}}{|F|}.$$
\end{lem}

\begin{proof}
Without loss of generality, we may assume that $X=F$. By a direct calculation, we have
\begin{eqnarray*}
\|\chi_AP_F\chi_B\| &=& \sup_{\|v\|=\|w\|=1} \langle \chi_AP_F\chi_B v,w \rangle = \sup_{\|v\|=\|w\|=1} \langle P_F\chi_B v, P_F \chi_A w \rangle\\
&=& \sup_{\|v\|=\|w\|=1} \left\langle \sum_{f\in F}\left( \frac{1}{|F|}\sum_{a\in A}v(a) \right)\delta_f, \sum_{f\in F}\left( \frac{1}{|F|}\sum_{b\in B}w(b) \right)\delta_f \right\rangle\\
&=& \sup_{\|v\|=\|w\|=1} \frac{1}{|F|^2}\cdot |F| \cdot \left(\sum_{a\in A}v(a)\right) \left( \overline{\sum_{b\in B}w(b)} \right)\\
&\leq &  \sup_{\|v\|=\|w\|=1} \frac{1}{|F|}\sqrt{|A|}\sqrt{|B|}\cdot\|v\|\cdot \|w\| \\
&\leq & \frac{\sqrt{|A||B|}}{|F|},
\end{eqnarray*}
where the penultimate inequality follows from the Cauchy-Schwarz Inequality. On the other hand, it is easy to see that
$$\langle P_F \chi_B v, P_F \chi_A w \rangle = \frac{\sqrt{|A||B|}}{|F|},$$
where $v,w$ are the normalised characteristic functions of $A,B$, respectively.
\end{proof}

From the definition of quasi-locality and the previous lemma, we directly obtain the following:
\begin{prop}\label{prop: char for averaging projection being quasi-local}
Let $\{X_n\}_{n \in \N}$ be a sequence of finite metric spaces with $|X_n| \to \infty$ as $n\to\infty$. Let $X=\bigsqcup_{n\in\N}X_{n}$ be a coarse disjoint union, and let $P_X$ the averaging projection of the sequence $\{X_n\}_{n\in \N}$.
Then $P_X$ is quasi-local \emph{if and only if}
$$0=\lim_{R \to +\infty} \sup\left\{\frac{|A||B|}{|X_n|^2}: n\in\N, \, A,B \subseteq X_n,\, d(A,B) \geq R\right\}.$$
\end{prop}

\begin{rem}
  The limit above always exists, since the function
  $$R\mapsto
\sup\left\{\frac{|A||B|}{|X_n|^2}: n\in\N, \, A,B \subseteq X_n,\, d(A,B) \geq
  R\right\}
  $$
  is non-increasing and bounded from below by $0$.
\end{rem}

We now establish the geometric conditions equivalent to $P_X$ being quasi-local.

\begin{thm}\label{prop: expanderish condition}
Let $\{X_n\}_{n \in \N}$ be a sequence of finite metric spaces with $|X_n| \to \infty$ as $n\to\infty$. Let $X=\bigsqcup_{n\in\N}X_{n}$ be a coarse disjoint union, and let $P_X$ the averaging projection of the sequence. Then the following are equivalent:
\begin{enumerate}
  \item[(1)] $P_X$ is quasi-local;
  \item[(2)] for any $\alpha\in(0,\frac12]$, any $c\in (0,1)$, there exists $R>0$ such that for any $n \in \N$ and $A \subseteq X_n$ with $\alpha|X_n| \leq |A| \leq |X_n|/2$, we have $|\partial_R A| > c|A|$;
  \item[(3)] for any $\alpha\in(0,\frac12]$, there exists $c\in (0,1)$ and $R>0$ such that for any $n \in \N$ and $A \subseteq X_n$ with $\alpha|X_n| \leq |A| \leq |X_n|/2$, we have $|\partial_R A| > c|A|$.
\end{enumerate}
\end{thm}

\begin{proof}
\emph{``(1) $\Rightarrow$ (2)":} Suppose (2) fails. Then there exists $\alpha_0>0$ and $c_0\in (0,1)$ such that for any $R>0$, there exists $n \in \N$ and $A_n \subseteq X_n$ with $\alpha_0 |X_n| \leq |A_n| \leq |X_n|/2$, while $|\partial_R A_n| \leq c_0 |A_n|$. Now for any $R>0$, we have
$$X_n=(X_n\setminus \Nd_R(A_n)) \sqcup A_n \sqcup \partial_R A_n,$$
which implies
\begin{eqnarray*}
|X_n\setminus \Nd_R(A_n)| &=& |X_n| - |A_n| - |\partial_R A_n| \geq |X_n| - |A_n| - c_0|A_n| = |X_n| - (1+c_0) |A_n| \\
&\geq &|X_n| - \frac{1+c_0}{2} |X_n| = \frac{1-c_0}{2}|X_n|.
\end{eqnarray*}
Hence we have
$$\frac{|A_n| \cdot |X_n\setminus \Nd_R(A_n)|}{|X_n|^2} \geq \frac{\alpha_0 |X_n| \cdot \frac{1-c_0}{2}|X_n|}{|X_n|^2} = \frac{\alpha_0(1-c_0)}{2} >0,$$
which implies that the limit is bounded from below:
$$\lim_{R \to +\infty} \sup\left\{\frac{|A||B|}{|X_n|^2}: n\in\mathbb{N},\, A,B
  \subseteq X_n,\, d(A,B) \geq R\right\} \geq \frac{\alpha_0(1-c_0)}{2} >0.$$
This is a contradiction to the assumption that $P_X$ is quasi-local by Proposition \ref{prop: char for averaging projection being quasi-local}.

\emph{``(2) $\Rightarrow$ (3)":} This is clear.

\emph{``(3) $\Rightarrow$ (1)":} Suppose $P_X$ is not quasi-local. Then by Proposition \ref{prop: char for averaging projection being quasi-local}, we know that
$$\alpha:=\frac{1}{2}\lim_{R \to +\infty} \sup\left\{\frac{|A||B|}{|X_n|^2}: n\in\N,\, A,B \subseteq X_n,\, d(A,B) \geq R\right\} >0.$$
Hence there exists an increasing sequence of natural numbers $\{m_n\}_{n \in \N}$ going to infinity, and $A_n,B_n \subseteq X_{m_n}$ with $d(A_n,B_n) >2n$ such that $|A_n|\cdot|B_n| \geq \alpha \cdot |X_{m_n}|^2$. Since $|A_n| \leq |X_{m_n}|$ and $|B_n| \leq |X_{m_n}|$, we obtain that $|A_n| \geq \alpha|X_{m_n}|$ and $|B_n| \geq \alpha |X_{m_n}|$.

By condition (3), for the above $\alpha$ there exists $c_0 \in (0,1)$ and $R_0 >0$ such that for any $n \in \N$ and $A' \subseteq X_n$ with $\alpha|X_n| \leq |A'| \leq |X_n|/2$, we have $|\partial_{R_0} A'| > c_0|A'|$. Now consider $\Nd_n(A_n)$ and $\Nd_n(B_n)$. First, since $n\leq m_n$, we have $\Nd_n(A_n),\Nd_n(B_n)\subseteq X_{m_n}$. Since $d(A_n,B_n) >2n$, we know that $\Nd_n(A_n)$ and $\Nd_n(B_n)$ are disjoint. Without loss of generality, we may assume that $|\Nd_n(A_n)| \leq |X_{m_n}|/2$. Hence, we have
$$\alpha |X_{m_n}| \leq |A_n| \leq |\Nd_n(A_n)| \leq |X_{m_n}|/2.$$
By induction, we have
$$|\Nd_n(A_n)| \geq (1+c_0)|\Nd_{n-R_0}(A_n)| \geq \ldots \geq (1+c_0)^{\lfloor n/R_0 \rfloor} |A_n|$$
for any $n \in \N$. Hence, we have
$$|X_{m_n}| \geq |\Nd_n(A_n)| \geq (1+c_0)^{\lfloor n/R_0 \rfloor} |A_n| \geq (1+c_0)^{\lfloor n/R_0 \rfloor} \alpha |X_{m_n}|$$
for any $n \in \N$, which is a contradiction.
\end{proof}

It is clear from Definition \ref{defn: expander graphs} that for a sequence of expander graphs, condition (3) in the above proposition holds. Hence, we introduce the following notion:

\begin{defn}\label{def: expanderish condition}\footnote{After we have already finished the paper, we were informed by R. Grigorchuk that a different notion of asymptotic expanders was introduced by him \cite[Definition~10.3]{G11}. As far as we see, these two notions are not directly related to each other.}
A sequence of finite metric spaces $\{X_n\}_{n \in \N}$ such that $|X_n| \to \infty$ (as $n\to\infty$) is said to be a sequence of \emph{asymptotic expanders} (or \emph{asymptotic expander graphs} when all $X_n$ are graphs) if for any $\alpha>0$, there exist $c\in (0,1)$ and $R>0$ such that for any $n \in \N$ and $A \subseteq X_n$ with $\alpha|X_n| \leq |A| \leq |X_n|/2$, we have $|\partial_R A| > c|A|$.
\end{defn}

By Example \ref{Ex: non-expander but expanderish} and Theorem~\ref{prop: expanderish condition}, there \emph{do} exist asymptotic expander graphs which are \emph{not} expander graphs:

\begin{cor}\label{cor: expanderish but non-expander}
The sequence $\{Y_n\}_{n \in \N}$ constructed in Example \ref{Ex: non-expander but expanderish} is a sequence of asymptotic expander graphs, but not expander graphs.
\end{cor}

\subsection{Coarse Invariance}
In this subsection, we prove that being a sequence of asymptotic expanders is a coarse property, i.e.~that it is invariant under coarse equivalence, provided that the equivalence respects the pieces. (Note that this is automatic when the terms of the sequence are sufficiently connected, e.g.~when they are connected graphs.)

Note that ``being a sequence of expander graphs'' is preserved under coarse equivalences. This is well-known to experts and a detailed proof can be found, for example, in \cite[Lemma 2.7.5]{Vig18} and \cite[Lemma~12]{S17}.

\begin{rem}
  In general, ``being a sequence of asymptotic expanders'' is \emph{not} a coarse property of the coarse disjoint union. Consider the following example: let $\{W_{n}\}_{n\in\N}$ be a sequence of expander graphs with $n/|W_n| \to 0$, and $\{Z_{n}\}_{n\in\N}$ be a sequence of Cayley graphs of $\Z/n\Z$ (with respect to the image of the generator $1\in\Z$). Next, let $X_{n} := W_{n}\sqcup Z_{n}$ for each $n\in\N$, where we declare the metric $d_{n}$ on $X_{n}$ to agree with the existing metrics on $W_{n}$, $Z_{n}$, and such that $d(W_{n},Z_{n})=2n$. Let $Y_{2k} := W_{k}$ and $Y_{2k-1}:=Z_{k}$ for $k\in\N$.

  Now consider coarse disjoint unions $X=\bigsqcup_{n\in\N}X_{n}$ and
  $Y=\bigsqcup_{n\in\N}Y_{n}$. Then $X$ and $Y$ are (bijectively) coarsely
  equivalent (essentially via assembling the identity maps on each $W_{n}$ and
  $Z_{n}$). However the sequence $\{X_{n}\}_{n\in\N}$ is a sequence of
  asymptotic expanders (this follows from a computation analogous to Example
  \ref{Ex: non-expander but expanderish}), but $\{Y_{n}\}_{n\in\N}$ is not.
  Indeed, the condition (2) of Theorem \ref{prop: expanderish condition} fails:
  Let $\alpha=\frac12$ and $c=\frac12$. Given $R\geq 1$, take $n = 8R$, we
  choose the subset $A_{n}\subset Y_{2n-1}=\Z/n\Z$ to be the image of
  $\{1,2,\dots,4R\}\subset\Z$. Then $|A_{n}|=|Y_{2n-1}|/2=\alpha|Y_{2n-1}|$ and
  $|\partial_{R}A_{n}|=2R = c|A_{n}|$.

  The issue here is that the coarse equivalence between $X$ and $Y$ does not map
  ``pieces to pieces''. This is explicitly prevented in the following theorem.
  Another way to circumvent this issue is to assume that the pieces are
  sufficiently connected, e.g.~connected graphs (see Corollary \ref{cor:coarse-invariance-ish-when-connected}).
\end{rem}

\begin{thm} \label{thm:coarse-invariance-of-ish}
Let $\{X_n\}_{n \in \N}$ and $\{Y_n\}_{n \in \N}$ be sequences of finite
metric spaces, such that $|X_n|, |Y_n| \to \infty$ as $n\to \infty$, and coarse
disjoint unions $X=\bigsqcup_{n\in\N}X_{n}$ and $Y=\bigsqcup_{n\in\N}Y_{n}$ have
bounded geometry.
Let $\varphi_{n}:X_{n}\to Y_{n}$, for $n\in\N$, be functions such that $\varphi=\bigsqcup_{n\in\N}\varphi_{n}:X\to Y$ is a coarse equivalence.
Then if  $\{X_n\}_{n \in \N}$ is a sequence of asymptotic expanders, then so is $\{Y_n\}_{n \in \N}$.
\end{thm}

Before embarking on the proof of Theorem \ref{thm:coarse-invariance-of-ish}, let us point out that under a suitable connectedness assumption, we can simplify the statement:

\begin{cor}\label{cor:coarse-invariance-ish-when-connected}
Let $\{X_n\}_{n \in \N}$ and $\{Y_n\}_{n \in \N}$ be sequences of finite
metric spaces, such that $|X_n|, |Y_n| \to \infty$ as $n\to \infty$, and coarse
disjoint unions $X=\bigsqcup_{n\in\N}X_{n}$ and $Y=\bigsqcup_{n\in\N}Y_{n}$ have
bounded geometry.
Assume that there exists a $D\geq0$, such that $X_{n}$ and $Y_{n}$ are $D$-connected for each $n\in \N$.
Then if $X$ and $Y$ are coarsely equivalent, and $\{X_{n}\}_{n\in\N}$ is a sequence of asymptotic expanders, then so is $\{Y_{n}\}_{n\in\N}$.
\end{cor}

\begin{proof}[Proof of Corollary \ref{cor:coarse-invariance-ish-when-connected}]
Let $f:X\to Y$ be a coarse equivalence and denote $f_{n}=f|_{X_{n}}$. Then by the
proof of \cite[Lemma 1]{KV17}, $f$ builds up a bijection between a co-finite
subset of $\{X_{n}\}_{n\in\N}$ and a co-finite subset of $\{Y_{n}\}_{n\in\N}$.
(Note that \cite{KV17} assumes that the pieces $X_{n}$ and $Y_{n}$ are connected graphs;
however the argument for the above claim carries through with obvious changes
when the pieces are assumed to be $D$-connected.)
Consequently, there exists $N\in\N$, such that for all $n\geq N$ there exists
$k(n)\in\N$, such that $f_{n}:X_{n}\to Y_{k(n)}$, and $\N\setminus\{k(n) : n\geq
N\}$ is finite.

Now $\{X_{n}\}_{n\in\N}$ is a sequence of asymptotic expanders if and only if
$\{X_{n}\}_{n\geq N}$ is, and likewise for $\{Y_{n}\}_{n\in\N}$ and
$\{Y_{k(n)}\}_{n\geq\N}$. The proof is finished by applying Theorem
\ref{thm:coarse-invariance-of-ish} to the latter two sequences.
\end{proof}

We now turn to the proof of Theorem \ref{thm:coarse-invariance-of-ish}, which
will be split into
several lemmas. First, we fix some notation. We shall use $\partial^{\rm
in}_R(A) = \{x\in A: d(x,A^c) \leq R\}$ for the \emph{inner $R$-boundary} of a
set $A\subset X$. We shall say that a subset $S\subset X$ is \emph{$D$-dense} in
$X$ (for some $D\geq 0$), if any point of $X$ is within distance $D$ of some
point in $S$. Finally, recall that if $X$ has bounded geometry, we denote
$N_X(R) = \sup_{x\in X}|B(x,R)|$.

\begin{lem} \label{lem:coarse-invariance-1}
Let $\psi: X\to Y$ be a function between two finite metric spaces, such that
$\psi(X)$ is $D$-dense in $Y$ for some $D\geq 0$. Let $B\subseteq Y$. Then
\begin{equation}\label{eq:pf-coarse-invariance-1}
|\psi^{-1}(B)| \geq |B|\cdot\frac{1}{N_Y(D)}\cdot\left(1-\frac{|\partial^{\rm in}_D(B)|}{|B|}\right).
\end{equation}
\end{lem}

\begin{proof}
Define $I:=\{y\in B : B(y,D)\subseteq B\}$. Given $y\in I$, there exists $x_y\in X$ such that $d(\psi(x_y),y)\leq D$, because $\psi(X)$ is $D$-dense in $Y$. Hence $\psi(x_y)\in B$, so $x_y\in \psi^{-1}(B)$. This defines an assignment $I\ni y\mapsto x_y\in \psi^{-1}(B)$, with at most $N_Y(D)$ elements of $I$ mapping to the same $x_y\in \psi^{-1}(B)$ (the points in $B(\psi(x_y),D)$). Hence $|\psi^{-1}(B)|\geq |I|/N_Y(D)$.

Next, if $y\in B\setminus I$, then $y\in \partial^{\rm in}_D(B)$, so $|I|\geq |B|-|\partial^{\rm in}_D(B)|$. Combining the two inequalities yields the desired one.
\end{proof}

\begin{lem} \label{lem:coarse-invariance-2}
Under the assumptions of Lemma \ref{lem:coarse-invariance-1}, assume further that $|B|\leq |Y|/2$, and that $1-N_Y(D)\frac{|\partial_D(B)|}{|B|}\geq\frac12$. Then either $A=\psi^{-1}(B)$ or $A=\psi^{-1}(Y\setminus B)$ satisfies $|A|\leq |X|/2$ and $|A|\geq \frac{1}{2N_Y(D)}|B|$.
\end{lem}

\begin{proof}
Suppose that $|\psi^{-1}(B)|\leq |X|/2$. We apply Lemma \ref{lem:coarse-invariance-1} with $B$ and let $A:=\psi^{-1}(B)$. Inserting the general observation $|\partial^{\rm in}_D(B)|\leq N_Y(D)|\partial_D(B)|$ into \eqref{eq:pf-coarse-invariance-1} and applying the assumption immediately yield $|A|\geq\frac{1}{2N_Y(D)}|B|$.

On the other hand, if $|\psi^{-1}(B)|>|X|/2$, we let $A:=\psi^{-1}(Y\setminus B)=X\setminus\psi^{-1}(B)$ (as then $|A|\leq |X|/2$).  Note that as $|B|\leq |Y|/2$, we have $|Y\setminus B|\geq |B|$. Furthermore, note that $1-\frac{|\partial_D(B)|}{|B|} \geq 1-N_Y(D)\frac{|\partial_D(B)|}{|B|} \geq \frac12$.

We apply Lemma \ref{lem:coarse-invariance-1} with $Y\setminus B$ in place of $B$. From \eqref{eq:pf-coarse-invariance-1} we then get
\begin{align*}
  |A| &= |\psi^{-1}(Y\setminus B)| \geq |Y\setminus B|\cdot\frac{1}{N_Y(D)}\cdot\left(1-\frac{|\partial^{\rm in}_D(Y\setminus B)|}{|Y\setminus B|}\right)\\
  &\geq |B|\cdot\frac{1}{N_Y(D)}\cdot\left(1-\frac{|\partial_D(B)|}{|B|}\right) \geq \frac{1}{2N_Y(D)}|B|,
\end{align*}
which finishes the proof.
\end{proof}

\begin{lem} \label{lem:coarse-invariance-3}
With the notation as in Lemma \ref{lem:coarse-invariance-1}, assume further that
$\psi$ is at most $K$-to-one (for some $K\geq1$). Furthermore, let $\rho_+$ be any
function which for all $u,v\in X$ satisfies $d(\psi(u),\psi(v))\leq
\rho_+(d(u,v))$. Then for any $S\geq 0$ both $A=\psi^{-1}(B)$ and
$A=\psi^{-1}(Y\setminus B)=X\setminus \psi^{-1}(B)$ satisfy $|\partial_S(A)|\leq
KN_X(S)|\partial_{\rho_+(S)}(B)|$.
\end{lem}

\begin{proof}
Denote $C=\psi^{-1}(B)$. Let $u\in \partial_S(C)$. Then there is a $v\in C$ with $d(u,v)\leq S$. Hence $\psi(v)\in B$, $\psi(u)\not\in B$ and $d(\psi(u),\psi(v))\leq \rho_+(S)$. In other words, $\psi(u)\in\partial_{\rho_+(S)}(B)$ and as $\psi$ is at most $K$-to-one, we get $|\partial_S(C)|\leq K|\partial_{\rho_+(S)}(B)|$.

When $A=\psi^{-1}(B)=C$, then the above inequality trivially implies $|\partial_S(A)|\leq KN_X(S)|\partial_{\rho_+(S)}(B)|$.

When $A=\psi^{-1}(Y\setminus B) = X\setminus C$, we use the general fact that $|\partial_S(X\setminus C)|\leq N_X(S)|\partial_S(C)|$. Now the inequality established in the first paragraph yields $|\partial_S(A)|\leq N_X(S)|\partial_X(C)|\leq KN_X(A)|\partial_{\rho_+(S)}(B)|$.
\end{proof}

\begin{proof}[Proof of Theorem \ref{thm:coarse-invariance-of-ish}]
  Recall that we have, for each $n\in\N$, a function $\varphi_{n}:X_{n}\to Y_{n}$,
  such that $\varphi=\bigsqcup_{n\in\N}\varphi_{n}$ is a coarse equivalence
  between $X=\bigsqcup_{n\in\N}X_{n}$ and $Y=\bigsqcup_{n\in\N}Y_{n}$.
  Since both have bounded geometry, the functions $N_X$ and $N_Y$ work in
  particular for any $X_n$ or $Y_n$, respectively.
  As $\varphi$ is a coarse equivalence, there exist functions $\rho_{\pm}:
  \mathbb{R}^+ \to \mathbb{R}^+$ such that $\rho_{-}(t)\to\infty$ as
  $t\to\infty$ and
  $$\rho_-(d(x,y)) \leq d(\varphi_{n}(x),\varphi_{n}(y)) \leq \rho_+(d(x,y))$$
  for any $n\in\N$ and $x,y \in X_{n}$.
  Without loss of generality, we may assume that both $\rho_{+}$ and $\rho_{-}$
  are non-decreasing.
  Moreover, there exists a $D\geq 0$ such that $\varphi_{n}(X_{n})$ is $D$-dense
  in $Y_{n}$.
  As $X$ has bounded geometry, it follows that there exists $K>0$ such that
  $|\varphi^{-1}(y)|\leq K$ for any $y \in Y$.

  Assume that $\{Y_n\}$ is \emph{not} a sequence of asymptotic expanders, i.e.
  there exists some $\alpha\in(0,\frac12]$ such that for any $R>0$, there exist
  sequences $\{k_{n}\}_{n\in\N}$ and $\{B_n\}_{n\in\N}$ with $B_n\subseteq
Y_{k_n}$ and $\alpha|Y_{k_n}| \leq |B_n| \leq |Y_{k_n}|/2$, such that
$|\partial_R B_n|/|B_n| \to 0$.

  Given any $S\geq 0$, take $R>\max\{\rho_+(S),D\}$ and let $\{B_{n}\}_{n\in\N}$
  be as described above. Without loss of generality, we can assume that
  $k_{n}=n$. Then $\partial_D(B_n)\subseteq \partial_R(B_n)$, and so also
  $|\partial_D(B_n)|/|B_n|\to 0$; likewise $|\partial_{\rho_+(S)}(B_n)|\leq
  |\partial_R(B_n)|$. Thus for sufficiently large $n$ the assumptions of Lemma
  \ref{lem:coarse-invariance-2} are satisfied by $\varphi_n:X_n\to Y_n$ and
  $B_n$, so we get a sequence of subsets $A_n\subseteq X_n$ with
  $\frac{1}{2N_Y(D)}|B_n|\leq |A_n|\leq |X_n|/2$. By Lemma
  \ref{lem:coarse-invariance-3}, they also satisfy $|\partial_S(A_n)|\leq
  KN_X(S)|\partial_R(B_n)|$.

  Since $\varphi_n$ is at most $K$-to-one, we get
  $$
  |A_n|\geq \frac{1}{2N_Y(D)}|B_n| \geq \frac{\alpha}{2N_Y(D)}|Y_n| \geq \frac{\alpha}{2KN_Y(D)}|X_n|,
  $$
  so the cardinalities of $A_n$ are at least a uniform proportion of $X_n$.
  Finally,
  $$
  \frac{|\partial_S(A_n)|}{|A_n|} \leq \frac{KN_X(S)|\partial_R(B_n)|}{\frac{1}{2N_Y(D)}|B_n|} \to 0.
  $$
  Thus we have shown that $\{X_n\}$ is not a sequence of asymptotic expanders either.
\end{proof}

\section{asymptotic expander graphs are not uniformly locally amenable} \label{sec:ula}

In this section, we show that being a sequence of asymptotic expanders leads to the failure of uniform local amenability. Recall from \cite[Proposition 3.2]{BNSW13} that Property A implies uniform local amenability, while the property of coarse embeddability into Hilbert spaces does not imply it generally (see \cite[Corollary 4.3]{BNSW13}). First let us recall the definition:
\begin{defn}[{\cite[Definition~2.2]{BNSW13}}]\label{def: ULA}
A metric space $(X,d)$ is said to be \emph{uniformly locally amenable} (ULA) if for all $R, \varepsilon>0$ there exists $S>0$ such that for any finite subset $F$ of $X$, there exists $E \subseteq X$ with $\diam(E) \leq S$ and
$|\partial_R(E) \cap F| < \varepsilon |E \cap F|$.
\end{defn}

Note that replacing $E$ with $E\cap F$, we can assume that $E \subseteq F$ in the above definition. We want to use another equivalent form of ULA as follows. For a finite subset $F \subseteq X$, define the \emph{associated normalised characteristic measure} $\mu_F$ to be
$$\mu_F(E):=\frac{|E \cap F|}{|F|}$$
for any $E \subseteq X$. Clearly, $\mu_F$ is a probability measure with finite support $F$. Then we can translate ULA in the following language directly:

\begin{lem}\label{lem: ULA equiv}
A metric space $(X,d)$ is uniformly locally amenable \emph{if and only if} for all $R, \varepsilon>0$ there exists $S>0$ such that for any finite $F \subseteq X$, there exists $E \subseteq F$ with $\diam(E) \leq S$ and $\mu_F(\partial_R E) < \mu_F(E)$.
\end{lem}

By the same argument as in \cite[Theorem 3.8]{BNSW13}, ULA implies a weaker
version of the metric sparsification property introduced by Chen, Tessera, Wang and Yu \cite{CTWY08} as follows:

\begin{lem}\label{lem: ULA implies wMSP}
Let $(X,d)$ be a metric space with ULA. Then for any $c\in (0,1)$ and $R>0$,
there exists $S>0$ such that for any finite $F\subseteq X$, there exists $\Omega
\subseteq F$ with a decomposition $\Omega=\bigsqcup_{i\in I} \Omega_i$ satisfying the
following:
\begin{itemize}
   \item $\mu_F(\Omega) \geq c$;
   \item $\diam(\Omega_i) \leq S$;
   \item $d(\Omega_i, \Omega_j) > R$ for $i \neq j$.
\end{itemize}
\end{lem}

\begin{proof}
Given $c\in (0,1)$ and $R>0$, take $\varepsilon=1/c-1$. By Lemma \ref{lem: ULA equiv}, there exists $S>0$ satisfying the condition therein. Given a finite subset $F\subseteq X$, we set $F_1:=F$. By assumption, there exists $E_1 \subseteq F_1$ with $\diam(E_1) \leq S$ and $\mu_F(\partial_R E_1) < \varepsilon\mu_F(E_1)$.

Now set $F_2:= F_1 \setminus \Nd_R(E_1)$. By assumption, there exists $E_2 \subseteq F_2$ with $\diam(E_2) \leq S$ and $\mu_{F_2}(\partial_R E_2) <\varepsilon \mu_{F_2}(E_2)$. Hence $|\partial_R E_2 \cap F_2|<\varepsilon |E_2|$, which implies $\mu_F(\partial_R E_2 \cap F_2) < \varepsilon\mu_F(E_2)$ since $F_2 \subseteq F$.

Similarly, we may set $F_3:=F_2\setminus \Nd_R(E_2)$ and continue the process. Since $F_1$ is finite, it must eventually terminate, providing two sequences $F_1 \supseteq F_2 \supseteq \ldots \supseteq F_n$ and $E_1, E_2, \ldots, E_n$ such that $E_i \subseteq F_i$ for all $i$ and
\begin{itemize}
  \item $\diam(E_i) \leq S$ for all $i$;
  \item $d(E_i,E_j)>R$ for $i \neq j$;
  \item $\mu_F(\partial_R E_i \cap F_i) < \varepsilon \mu_F(E_i)$.
\end{itemize}
Set $\Omega_i:=E_i$ and $\Omega:=\bigsqcup_{1\leq i\leq n} \Omega_i$. We have
$$1=\mu_F(F_1)=\sum_{i=1}^n \mu_F(E_i) + \mu_{F}(\partial_R E_i \cap F_i) < \sum_{i=1}^n (1+\varepsilon)\mu_F(E_i) = (1+\varepsilon)\mu_F(\Omega),$$
which implies that $\mu_F(\Omega) > 1/(1+\varepsilon)=c$. So we finish the proof.
\end{proof}

\begin{thm}\label{thm: expanderish implies non ULA}
Let $X$ be a metric space with bounded geometry, which is a coarse disjoint union of a sequence of asymptotic expanders. Then $X$ is not uniformly locally amenable. In particular, $X$ does not have Property A.
\end{thm}

\begin{proof}
By assumption, we can write $X=\bigsqcup_{n\in\mathbb{N}}X_n$, where
$\{X_n\}_{n\in\mathbb{N}}$ is a sequence of asymptotic expanders.
Assume that $X$ is uniformly locally amenable. Setting $c=1/2$ and given $R>0$,
by Lemma \ref{lem: ULA implies wMSP} there exists $S=S(R)>0$ such that for any
finite subset $F\subseteq X$, there exists $\Omega \subseteq F$ with a
decomposition $\Omega=\bigsqcup_{i\in I} \Omega_i$ satisfying the conditions therein.
Hence for
each $n \in \N$, there exists $\Omega^{(n)} \subseteq X_n$ with a decomposition
$\Omega^{(n)} = \bigsqcup_{i\in I} \Omega^{(n)}_i$ satisfying:
\begin{itemize}
   \item $|\Omega^{(n)}| \geq |X_n|/2$;
   \item $\diam(\Omega^{(n)}_i) \leq S(R)$;
   \item $d(\Omega^{(n)}_i, \Omega^{(n)}_j) > R$ for $i \neq j$.
\end{itemize}
As $X$ has bounded geometry, $N_{X}(S)=\sup_{x\in X}|B(x,S)|$ is finite.
Since each $\Omega_i^{(n)}$ has cardinality at most $N_{X}(S)$, for any $n\in\N$
we may take a decomposition $I=I_1 \sqcup I_2$ such that for
$$A_n:=\bigsqcup_{i \in I_1} \Omega_i^{(n)} \subseteq X_n \quad \mbox{and} \quad B_n:=\bigsqcup_{i \in I_2} \Omega_i^{(n)} \subseteq X_n,$$
we have $A_n \sqcup B_n = \Omega^{(n)}$ and
$$|A_n|, |B_n| \in \left[\frac{|\Omega^{(n)}|}{2} - N_{X}(S) , \frac{|\Omega^{(n)}|}{2} + N_{X}(S)\right].$$
Note that by construction we have $d(A_n,B_n) \geq R$.
For the given $R$ (and thus also $S=S(R)$), we may choose $n$ sufficiently large such that
$\frac{N_{X}(S)}{|X_n|} < \frac{1}{32}$. Consequently
$$
\frac{|A_n|\cdot |B_n|}{|X_n|^2}
\geq \frac{\left(|\Omega^{(n)}|/2 - N_{X}(S)\right)^2}{|X_n|^2}
\geq \frac{|\Omega^{(n)}|^2}{4|X_n|^2} - \frac{|\Omega^{(n)}|\cdot
  N_{X}(S)}{|X_n|^2}
\geq \frac{1}{16} - \frac{N_{X}(S)}{|X_{n}|}
\geq \frac{1}{32}.
$$

In conclusion, for any $R>0$, we obtain $n\in\N$ and sets $A_n,B_n \subseteq
X_n$, such that $d(A_n,B_n) \geq R$ and $\frac{|A_n|\cdot |B_n|}{|X_n|^2} \geq
\frac{1}{32}$. This is a contradiction with the condition in Proposition \ref{prop: char for averaging projection being quasi-local}, so $X$ is not uniformly locally amenable. Finally recall from \cite[Proposition 3.2]{BNSW13} that Property A implies uniform local amenability, so $X$ does not have Property A and we finish the proof.
\end{proof}

\section{Nuclearity of uniform quasi-local algebras} \label{sec:nuclearity}

From \cite[Theorem~5.3]{STY02} and Proposition~\ref{prop: SZ result} we know that the uniform quasi-local algebra $\Cq(X)$ is nuclear for every metric space with bounded geometry and Property A. In this section, we provide a proof for the converse implication: the nuclearity of the uniform quasi-local algebra $\Cq(X)$ implies that $X$ has Property A.

Firstly, let us recall some related notions and facts:

\begin{defn}[{\cite[Definition 2.1.1 and Definition 2.3.1]{BO08}}]\label{def: nuclear map}
Let $\A$ and $\mathcal{B}$ be two $C^*$-algebras. A map $\theta: \A \to \mathcal{B}$ is called \emph{nuclear} if for any $\varepsilon>0$ and any finite subset $F \subseteq \A$, there exist $n \in \N$ and contractive completely positive maps $\varphi: \A \to M_n(\C)$ and $\psi: M_n(\C) \to \mathcal{B}$ such that $\|\psi \circ \varphi (a) -\theta(a)\|< \varepsilon$ for any $a \in F$. A $C^*$-algebra $\A$ is called \emph{nuclear} if the identity map $\Id_{\A}$ is nuclear.
\end{defn}

\begin{prop}[{\cite[Theorem 5.3]{STY02}}]\label{prop: characterisation of property A via nuclearity}
Let $(X,d)$ be a metric space with bounded geometry, then $X$ has Property A \emph{if and only if} the uniform Roe algebra $C^*_u(X)$ is nuclear.
\end{prop}

We need the following auxiliary lemma characterising Property A, which is a slight modification of Proposition \ref{prop: characterisation of property A}. The proof is elementary, hence we leave it to the readers.

\begin{lem}\label{lem: strong ce}
Let $(X, d)$ be a metric space with bounded geometry. Then the following are equivalent:
\begin{enumerate}
  \item $(X, d)$ has Property A.
  \item For any $R,\varepsilon > 0$ there exist a map $\eta: X \to \ell^2(X)$ satisfying:
\begin{enumerate}
  \item $\|\eta_x\|_2=1$ for every $x \in X$;
  \item for $x,y \in X$ with $d(x, y) < R$, we have $\|\eta_x-\eta_y\|_2 < \varepsilon$;
  \item $\lim\limits_{S\to \infty} \sup\limits_{x\in X} \sum\limits_{z\notin B(x,S)} |\eta_x(z)|^2=0$.
\end{enumerate}
\end{enumerate}
\end{lem}

Recall that an operator $T \in \B(\ell^2(X))$ is called a ghost operator if for any $\varepsilon>0$, there exists a bounded subset $B \subseteq X$ such that for any $x,y \in X \setminus B$, we have $|T_{x,y}|< \varepsilon$. It is easy to check that all the ghost operators in the uniform Roe algebra $C^*_u(X)$ form an ideal in $C^*_u(X)$. The same situation also holds in the case of uniform quasi-local algebra $\Cq(X)$:

\begin{lem}\label{lem: ghost ideal in quasi-local algebra}
All the ghost operators in the uniform quasi-local algebra $\Cq(X)$ form a two-sided closed ideal in $\Cq(X)$.
\end{lem}

\begin{proof}
Let $G, T \in \Cq(X)$ with norm $1$ and suppose $G$ is ghost. It suffices to show that $GT$ and $TG$ are ghost as well. We only prove the case of $GT$, while the case of $TG$ is similar.

Fix an $\varepsilon>0$. Since $T$ is quasi-local, there exists a $R>0$ such that for any $A,B \subseteq X$ with $d(A,B)>R$, then $\|\chi_A T \chi_B\| < \varepsilon$. Hence for any $y\in X$, we have
$$\|\chi_{B(y,R)^c} T \chi_{\{y\}}\| = \Big( \sum_{z: d(z,y)>R} |T_{z,y}|^2\Big)^{\frac{1}{2}} < \varepsilon.$$
Taking $\varepsilon':=\frac{\varepsilon}{N_X(R)}$, there exists a finite subset $B$ such that for any $x',y'\in X\setminus B$, then $|G_{x',y'}|< \varepsilon'$. Now taking $B':= \Nd_R(B)$, and note that for any $y\notin B'$ and $z\in B(y,R)$, we have $z \notin B$. Hence for any $x,y\in X\setminus B'$, we have:
\begin{eqnarray*}
|(GT)_{x,y}| &\leq & \Big| \sum_{z: d(z,y) \leq R} G_{x,z} T_{z,y} \Big| + \Big| \sum_{z: d(z,y)>R} G_{x,z} T_{z,y} \Big| \\
&\leq & \Big(\sum_{z: d(z,y) \leq R} \varepsilon'\cdot |T_{z,y}|\Big) + \|G\| \cdot \Big( \sum_{z: d(z,y)>R} |T_{z,y}|^2\Big)^{\frac{1}{2}}\\
&\leq & \varepsilon' \cdot |B(y,R)| \cdot \|T\| + \|G\| \cdot \varepsilon\\
&\leq & 2\varepsilon,
\end{eqnarray*}
where we use the Cauchy-Schwartz Inequality in the second inequality. Since the closeness is clear, we finish the proof.
\end{proof}

Now we are in the position to prove the following main result of this section, whose proof is inspired by that of \cite[Theorem~5.5.7]{BO08}.
\begin{thm}\label{thm: nuclearity of quasi-local algebra}
For a metric space $(X, d)$ with bounded geometry, the following are equivalent:
\begin{enumerate}
  \item $X$ has Property A;
  \item the uniform quasi-local algebra $\Cq(X)$ is nuclear;
  \item the canonical inclusion $C^*_u(X) \hookrightarrow \Cq(X)$ is nuclear;
  \item all ghost operators in the uniform quasi-local algebra $\Cq(X)$ are compact.
  \item $\ell^\infty(X)$ separates ideals of $\Cq(X)$. In other words, the closed ideal generated by $I\cap \ell^\infty(X)$ inside $\Cq(X)$ is equal to $I$ for every closed ideal $I$ in $\Cq(X)$.
\end{enumerate}
\end{thm}

\begin{proof}
From Proposition \ref{prop: SZ result}, we know that if $X$ has Property A then $C^*_u(X)=\Cq(X)$. Hence combining with Proposition \ref{prop: characterisation of property A via nuclearity}, we have ``(1) $\Rightarrow$ (2)"; and combining with Proposition \ref{prop: characterisation of property A}, we have ``(1) $\Rightarrow$ (4)". It follows directly from \cite[Theorem~3.20]{BL18} (see also \cite{CW04}) and Lemma \ref{lem: ghost ideal in quasi-local algebra} that ``(1) $\Rightarrow (5) \Rightarrow (4)$" holds.

On the other hand, since $C^*_u(X)$ is a subalgebra of $\Cq(X)$, we know that condition (4) implies that all ghost operators in $C^*_u(X)$ are compact. Hence from Proposition \ref{prop: characterisation of property A} again, we obtain ``(4) $\Rightarrow$ (1)". Also notice that due to the fact that the composition of two completely positive maps is nuclear provided either one of them is, we know that ``$(2) \Rightarrow$ (3)" holds.

Therefore, it suffices to prove ``$(3) \Rightarrow$ (1)". To summarise the rest of the proof, we follow \cite[Theorem~5.5.7]{BO08} to construct ``Property A'' vectors (Proposition \ref{prop: characterisation of property A}); but in the last step, instead of uniform bound on supports, we use quasi-locality to get strong summability, as in condition (c) in Lemma \ref{lem: strong ce}.

Assume that the inclusion $C^*_u(X) \hookrightarrow \Cq(X)$ is nuclear. Let $R>0$ and $\varepsilon>0$. Since $X$ has bounded geometry, there exists a finite set $\mathcal{F}$ of partial isometries in $\mathbb{C}_u [X]$ with the property that for any $x, y \in X$ with $d(x, y) \leq R$, there exists $v \in \mathcal{F}$ such that $v \delta_x = \delta_y$ (see e.g. \cite[Lemma 2.6]{SZ18}). Since the inclusion $C^*_u(X) \hookrightarrow \Cq(X)$ is nuclear, there exist unital completely positive maps $\phi : C^*_u(X) \to  M_n (\C)$ and $\psi : M_n (\C)\to  \Cq (X)$ such that $\|(\psi\circ \phi )(v) - v\|< \varepsilon$ for all $v \in \mathcal{F}$ (see also  \cite[Proposition~2.2.6]{BO08}).

Denoting by $\{e_{ij}\}_{1\leq i,j\leq n}$ the matrix units of $M_n(\C)$, the
matrix $[\psi(e_{ij})]_{i,j=1}^n$ is positive in $M_n(\Cq (X))$
\cite[Proposition 1.5.12]{BO08}. Let $[b_{ij}]= [\psi(e_{ij})]^{1/2}\in M_n(\Cq
(X))$ and denote by $\{\xi_i\}_{1\leq i\leq n}$ the standard basis for $\C^n$. We define
\[
\xi_{\psi}=\sum_{j,k=1}^n \xi_j\otimes \xi_k\otimes b_{kj}\in \C^n\otimes\C^n\otimes \Cq (X).
\]
Note that $\C^n\otimes\C^n\otimes \Cq (X)$ is a Hilbert $\Cq(X)$-module equipped with an inner product $\langle \cdot,\cdot\rangle$ defined by
\[
\langle \xi \otimes \eta \otimes T, \xi' \otimes \eta' \otimes T' \rangle=\langle \xi,\xi' \rangle_{\C^n} \cdot \langle \eta, \eta'\rangle_{\C^n} \cdot T^*T'
\]
for elementary tensor elements in $\C^n\otimes\C^n\otimes \Cq (X)$, where $\langle \cdot, \cdot \rangle_{\C^n}$ is the standard inner product on $\C^n$ which is linear in the second variable, then we extend it linearly to general elements in $\C^n\otimes\C^n\otimes \Cq (X)$. Note that $M_n(\C)\otimes M_n(\C)$ acts on (the first two tensor factors of) $\C^n\otimes\C^n\otimes \Cq(X)$. With this action, it is straightforward to check that for any $A\in M_n(\C)$ we have
\begin{equation} \label{eq:ozawa-nuclearity-trick}
\psi(A) = \langle \xi_{\psi},(A\otimes 1_{n}) \xi_{\psi}\rangle_{\Cq (X)},
\end{equation}
which implies that $\|\xi_{\psi}\|_{\Cq (X)}=1$ (choosing $A=1$). Denoting by $\{\eta_l\}_{1\leq l\leq n^2}$ the standard basis for $\C^{n^2}\cong \C^n\otimes\C^n$, we write
$$ \xi_{\psi}=\sum_{l=1}^{n^2} \eta_l\otimes a_l \in \C^{n^2}\otimes \Cq (X).$$
Now we define a map $\zeta : X \to  \ell^2 (X)$ by
\[
\zeta_x(z) =\big\|\sum_{l=1}^{n^2} \eta_l\langle \delta_z, a_l \delta_x  \rangle_{\ell^2(X)}\big\|_{\C^{n^2}}.
\]
We proceed analogously to the argument in the proof of \cite[Theorem 5.5.7]{BO08} to show that $\zeta$ satisfies the conditions (a) and (b) from Lemma \ref{lem: strong ce}(2). For the convenience of readers, we present the details here as well. For any $x \in X$, we have:
\begin{align*}
\|\zeta_x\|^2 &= \sum_z\big\|\sum_l \eta_l\langle \delta_z, a_l \delta_x  \rangle_{\ell^2(X)} \big\|^2_{\C^{n^2}}\\
&= \sum_{l,k} \sum_z \langle \eta_l, \eta_k \rangle_{\C^{n^2}} \overline{\langle \delta_z, a_l \delta_x \rangle_{\ell^2(X)}} \langle \delta_z, a_k \delta_x  \rangle_{\ell^2(X)}\\
&= \sum_{l,k} \langle \eta_l, \eta_k \rangle_{\C^{n^2}}  \langle a_l \delta_x, a_k \delta_x  \rangle_{\ell^2(X)}\\
&= \langle \delta_x, \langle \xi_\psi,\xi_\psi \rangle_{\C^{n^2}\otimes \Cq (X)}\delta_x \rangle_{\ell^2(X)} =1.
\end{align*}
On the other hand, let $d(x, y) \leq R$ and choose $v \in \mathcal{F}$ such that $v \delta_x =\delta_y$.
Using that $\phi(v)$ is contractive and \eqref{eq:ozawa-nuclearity-trick} with $A=\phi(v)$, we have that
\begin{align*}
\langle \zeta_y,\zeta_x\rangle &= \sum_z  \big\|\sum_l \eta_l\langle \delta_z, a_l \delta_y  \rangle\big\|_{\C^{n^2}} \cdot \big\|\sum_k \eta_k\langle \delta_z, a_k \delta_x  \rangle\big\|_{\C^{n^2}}\\
&\geq \sum_z  \big\|\sum_l \eta_l\langle \delta_z, a_l \delta_y  \rangle\big\|_{\C^{n^2}} \cdot \big\|(\phi(v)\otimes 1)\sum_k \eta_k\langle \delta_z, a_k \delta_x  \rangle\big\|_{\C^{n^2}}\\
&\geq \big|\sum_{l,k} \sum_z \langle \eta_l, (\phi(v)\otimes 1)\eta_k \rangle_{\C^{n^2}} \cdot  \overline{\langle \delta_z, a_l \delta_y  \rangle_{\ell^2(X)}} \cdot \langle \delta_z, a_k \delta_x \rangle_{\ell^2(X)}\big|\\
&=\big|\langle \delta_y, \langle \xi_{\psi}, (\phi(v) \otimes 1) \xi_{\psi}\rangle_{\C^n\otimes\C^n\otimes \Cq (X)} \delta_x\rangle_{\ell^2(X)}\big| = \big|\langle  \delta_y, (\psi\circ \phi )(v) \delta_x\rangle_{\ell^2(X)}\big|\\
& \geq \big|\langle  \delta_y, v \delta_x\rangle_{\ell^2(X)}\big| - \varepsilon = 1-\varepsilon.
\end{align*}
This implies that $\|\zeta_x-\zeta_y\|$ is sufficiently small.

We also claim:
$$ \lim_{S \to \infty} \sup_{x\in X} \sum_{y\notin B(x,S)}|\zeta_x(y)|^2 =0.$$
In fact, by definition we have
$$|\zeta_x(y)|^2 = \sum_{l=1}^{n^2}|\langle \delta_y,a_l \delta_x\rangle|^2.$$
Note that for all $1\leq l\leq n^2$, $a_l$ is quasi-local by assumption. Hence given any $\varepsilon'>0$, there exists an $S>0$ such that for all $x\in X$, we have
\[
\|\chi_{B(x,S)^c} a_l \chi_{\{x\}}\|^2 = \sum_{y\notin B(x,S)}\langle \delta_y, a_l\delta_x\rangle|^2 <\varepsilon'/ n^2,
\]
which implies that $\sup_{x\in X} \sum_{y\notin B(x,S)} |\zeta_x(y)|^2<\varepsilon'$.

In conclusion, we have shown that the function $\zeta$ constructed above satisfies also condition (c) from Lemma \ref{lem: strong ce}(2). Hence due to Lemma \ref{lem: strong ce} we are done.
\end{proof}

\begin{rem}
Let $X$ be a metric space with bounded geometry such that $X$ admits a coarse embedding into a countable discrete group. Then additionally, conditions (1) $\sim$ (5) in the above theorem are also equivalent to: (6) $\Cq(X)$ is exact. Indeed, this follows directly from \cite[Corollary~30]{BNW07} and the facts that nuclearity implies exactness, and exactness is preserved under taking $C^*$-subalgebras (we refer readers to \cite{BO08} for the relevant concepts).
\end{rem}

\begin{rem}
Let $G$ be a finitely generated residually finite group and $X$ be any of its box spaces. Then additionally, conditions (1) $\sim$ (5) in the above theorem are also equivalent to: (6) $\Cq(X)$ is exact; (7) $\Cq(X)$ is locally reflexive. Indeed, since nuclearity implies exactness and exactness implies locally reflexivity, it remains to prove condition (7) implies $X$ having Property A. Suppose not, then $X$ is a weak expander by \cite[Lemma~2.6]{Sak13}. In particular, the uniform Roe algebra $C^*_u(X)$ is not locally reflexive (see \cite[Theorem~1.1]{Sak13}). Since locally reflexivity is preserved under taking $C^*$-subalgebras, we conclude that $\Cq(X)$ is not locally reflexive as well (see \cite[Chapter 9]{BO08} for more details).

Finally, we record here that if the box space $X$ is a sequence of asymptotic expanders then $X$ must be a weak expander (see Theorem~\ref{thm: expanderish implies non ULA}, \cite[Theorem~4.5]{BNSW13} and \cite[Lemma~2.6]{Sak13}).
\end{rem}

\begin{rem}
Very recently, Sako proved in \cite{Sak19} a remarkable result that for a metric space $X$ with bounded geometry, $X$ has Property A \emph{if and only if} $C^*_u(X)$ is exact \emph{if and only if} $C^*_u(X)$ is locally reflexive. Therefore combining Theorem \ref{thm: nuclearity of quasi-local algebra} with Sako's result, we conclude that for a general metric space with bounded geometry, conditions (1) $\sim$ (7) above are all equivalent.
\end{rem}

\section{Cartan subalgebras in uniform quasi-local algebras} \label{sec:cartan}

The main result of this section is Proposition \ref{prop: two algebras are euqal}, which provides another take on the question when $C^*_u(X)=\Cq(X)$ in the context of Cartan subalgebras of these algebras.

Recall that a pair of $C^*$-algebras $\mathcal{B} \subseteq \A$ is a \emph{Cartan
pair} (or \emph{$\mathcal{B}$ is a Cartan subalgebra of $\mathcal{A}$}) \cite{Ren08} if $\mathcal{B}$ is a maximal abelian self-adjoint subalgebra
containing an approximate unit of $\A$ such that the normaliser of $\mathcal{B}$
inside $\A$ generates $\A$ as a $C^*$-algebra, and there exists a faithful
conditional expectation $E: \A \to \mathcal{B}$. Here the normaliser of
$\mathcal{B}$ in $\A$ is defined as $\{a\in \A: a\mathcal{B}a^* \cup a^*
\mathcal{B} a \subseteq \mathcal{B}\}$. It is clear that $\ell^\infty(X)
\subseteq \Cq(X)$ is a Cartan pair \emph{if and only if} the normaliser of
$\ell^\infty(X)$ in $\Cq(X)$ generates $\Cq(X)$. Moreover, $\ell^\infty(X)
\subseteq C^*_u(X)$ is always a Cartan pair.

\begin{prop}\label{prop: two algebras are euqal}
Let $X$ be a metric space with bounded geometry. Then the following are equivalent:
\begin{enumerate}
  \item $C^*_u(X)=\Cq(X)$;
  \item $\ell^\infty(X) \subseteq \Cq(X)$ is a Cartan pair;
  \item the normaliser of $\ell^\infty(X)$ in $\Cq(X)$ generates $\Cq(X)$.
\end{enumerate}
\end{prop}

To prove it, we need the following lemma analysing the normalisers of $\ell^\infty(X)$ in $C^*_u(X)$ and $\Cq(X)$:
\begin{lem}\label{lem: normaliser of the diagonal}
Let $X$ be a metric space with bounded geometry. Then the normalisers of $\ell^\infty(X)$ in $C^*_u(X)$ and in $\Cq(X)$ are the same. More precisely, the following are equivalent for $T\in \B(\ell^2(X))$:
\begin{enumerate}
  \item $T$ belongs to the normaliser of $\ell^\infty(X)$ in $C^*_u(X)$;
  \item $T$ belongs to the normaliser of $\ell^\infty(X)$ in $\Cq(X)$;
  \item $T=fV^\theta$ for some $f \in \ell^\infty(X)$ and some bijection $\theta: D \to R$ where $D, R \subseteq X$ satisfying: for any $\varepsilon>0$, there exists some $K>0$ such that for any $x\in D$ with $d(x,\theta(x))>K$ we have $|f(x)|<\varepsilon$.
\end{enumerate}
\end{lem}

\begin{proof}
It is straightforward to check that for $T=fV^\theta$ satisfying condition (3), we have $T\in C^*_u(X) \subseteq \Cq(X)$, and it normalises $\ell^\infty(X)$. This implies ``(3) $\Rightarrow$ (1)" and ``(3) $\Rightarrow$ (2)".

For the other directions, note that any element in the normaliser of $\ell^\infty(X)$ in $\B(\ell^2(X))$ has the form of $T=fV^\theta$ for some $f \in \ell^\infty(X)$ and bijection $\theta: D \to R$ for $D,R \subseteq X$.
Since $\|T\|=\sup_{x\in X}|f(x)|$, one can check directly that $T=fV^\theta$ is quasi-local if and only if for any $\varepsilon>0$, there exists some $K>0$ such that if $x\in D$ with $d(x,\theta(x))>K$ then $|f(x)|<\varepsilon$. The same condition also implies that $fV^\theta$ can be approximated in norm by operators with finite propagation using the fact that $\|T\|=\sup_{x\in X}|f(x)|$ again.
\end{proof}

\begin{proof}[Proof of Proposition \ref{prop: two algebras are euqal}]
Since $\ell^\infty(X) \subseteq C^*_u(X)$ is a Cartan pair, we have ``(1) $\Rightarrow$ (2) $\Rightarrow$ (3)". Now condition (3) says that $\Cq(X)$ is generated by the normaliser of $\ell^\infty(X)$ in $\Cq(X)$, which coincides with the normaliser of $\ell^\infty(X)$ in $C^*_u(X)$ by Lemma \ref{lem: normaliser of the diagonal}. Since $\ell^\infty(X) \subseteq C_u^*(X)$ is a Cartan pair, condition (1) holds.
\end{proof}

\section{Open questions}\label{sec:question}

According to Proposition~\ref{prop: SZ result} and Proposition~\ref{prop: two algebras are euqal}, we may ask the following natural question:
\begin{question}
Let $X$ be a metric space with bounded geometry. Suppose that $\ell^\infty(X)$ is a Cartan subalgebra of the uniform quasi-local algebra $\Cq(X)$. Does $X$ have Property A?
\end{question}

Let $\{Y_n\}_{n \in \N}$ be the sequence of asymptotic expander graphs in
Example~\ref{Ex: non-expander but expanderish} (see also Corollary~\ref{cor:
expanderish but non-expander}). Since the averaging projection $P_Y$ sits inside
the uniform Roe algebra $C^*_u(Y)$ for the coarse disjoint union $Y=\bigsqcup_{n\in\N}
Y_n$, it follows that $Y$ does not satisfy the coarse Baum-Connes conjecture
provided that $\{Y_n\}_{n \in \N}$ has large girth (see \cite{H99}, \cite{HLS02}
and \cite[Theorem~6.1]{MR2871145}). Does this conclusion hold generally?

\begin{question}
If $\{Y_n\}_{n \in \N}$ is any sequence of asymptotic expander graphs with large girth and let $Y$ be its coarse disjoint union, does the coarse Baum-Connes conjecture for $Y$ fail?
\end{question}

We now turn to the relation between asymptotic expanders and coarse embeddability.
It is well known that a sequence of expander graphs can not be coarsely embedded into any Hilbert space (see e.g. \cite[Theorem~5.6.5]{NY12}).
\begin{question}\label{ques}
Let $X$ be a coarse disjoint union of asymptotic expanders $\{X_n\}_{n\in \N}$ with bounded geometry. Can $X$ be coarsely embedded into some Hilbert space?
\end{question}

This question has a negative answer with an extra hypothesis:

\begin{prop}\label{cor: non CE}
Let $X$ be a coarse disjoint union of asymptotic expanders with bounded geometry. If $C^*_u(X)=\Cq(X)$, then $X$ can not be coarsely embedded into any Hilbert space.
\end{prop}
\begin{proof}
It follows from the hypothesis and Theorem~\ref{prop: expanderish condition}
that the averaging projection $P_X$ belongs to the uniform Roe algebra
$C^*_u(X)$. On the other hand, if $X$ can be coarsely embedded into a Hilbert
space, then the Roe algebra $C^*(X)$ (and hence also the uniform Roe algebra
$C^{*}_{u}(X)$) does not possess any non-compact ghost
projection by \cite[Proposition~35]{Fin14} and \cite[Theorem 1.1]{Yu00}. Since
$P_X$ is always a non-compact ghost projection, we complete the proof.
\end{proof}

If Question \ref{ques} has an affirmative answer (i.e., there exists a sequence of asymptotic expanders which can be coarsely embedded into some Hilbert space), then from Proposition~\ref{cor: non CE} we would provide an example of a space $X$ such that the uniform Roe algebra $C^*_u(X)$ is properly contained in the uniform quasi-local algebra $\Cq(X)$, which answers Question 6.7 in \cite{SZ18}.

\bibliographystyle{plain}
\bibliography{expanderish}

\end{document}